\documentclass[10pt,a4paper,reqno]{amsart}
\usepackage{amsfonts,amsthm,amsmath,amssymb,amscd,latexsym,url,microtype}

\newtheorem{theorem}{Theorem}[section]
\newtheorem{lemma}[theorem]{Lemma}
\newtheorem{corollary}[theorem]{Corollary}
\newtheorem{proposition}[theorem]{Proposition}
\newtheorem{conjecture}[theorem]{Conjecture}
\newtheorem{remark}[theorem]{Remark}

\newcommand{\C}{\mathbb C}
\newcommand{\R}{\mathbb R}

\newcommand{\diag}{\operatorname{diag}}

\begin{document}
\title[Reducibility of spectral curves of Jacobi pencils]{Reducibility of spectral curves of finite Jacobi pencils}
\author[B.~Shapiro]{Boris Shapiro}
\address{Department of Mathematics, Stockholm University, SE-106 91, Stockholm, Sweden}
\email{shapiro@math.su.se}
\date{}

\begin{abstract}
We study reducibility of the spectral polynomial
\[
\chi_n(\lambda,w)=\det(\lambda I+A+wB),
\]
where $A$ is diagonal and $B$ is tridiagonal with zero diagonal.  On the
principal stratum, where the diagonal entries are pairwise distinct and all
off-diagonal couplings are non-zero, we prove generic irreducibility for a
fixed diagonal and classify reducibility completely for $n\le4$.  In degree
four the reducible locus, expressed in the squared couplings $c_i=b_i^2$, is
smooth and irreducible of codimension two and admits an explicit
parameterization.

Two structural results hold in all degrees.  If $n$ is even, every polynomial
factor has even $\lambda$-degree.  If a degree-one factor occurs, then $n$ is
odd and the corresponding constant branch is forced to the central site; an
explicit mirror criterion characterizes it.  We also construct an
$(n+1)$-parameter endpoint family
\[
c_k=\kappa(a_{k+1}-a_1)(a_k-a_n)
\]
for which a quadratic factor peels off and the complementary factor is the
spectral polynomial of the chain with the two end sites deleted.  This realizes
codimension $n-2$ for the splitting $2+(n-2)$ in every degree and yields the
spin factorization by iteration.

Finally, we organize reducibility by specialization subsets at $w=0$.  A
coefficient-matching incidence count gives the natural bound
$\lceil d(n-d)/2\rceil$ for a factor of degree $d$; in degree eight exact
Jacobian computations realize the predicted codimensions $6$ and $8$ for
the splitting types $2+6$ and $4+4$.  On the real Jacobi locus we prove a
rank-symmetry constraint.  We conclude with reversal-invariance and
expected-codimension conjectures for the general reducible locus.
\end{abstract}

\maketitle

\section{Introduction}
Let
\[
J_n(w)=
\begin{pmatrix}
 a_1&wb_1&0&\cdots&0\\
 wb_1&a_2&wb_2&\ddots&\vdots\\
 0&wb_2&a_3&\ddots&0\\
 \vdots&\ddots&\ddots&\ddots&wb_{n-1}\\
 0&\cdots&0&wb_{n-1}&a_n
\end{pmatrix}
=A+wB
\]
be a finite Jacobi pencil, with complex parameters unless stated otherwise,
and set
\[
\chi_n(\lambda,w)=\det(\lambda I+J_n(w)).
\]
We call $\chi_n(\lambda,w)=0$ its affine spectral curve.  Since the
off-diagonal parameter enters only through pairs, $\chi_n$ is even in $w$;
we therefore write
\[
t=w^2,\qquad P_n(\lambda,t)=\chi_n(\lambda,w).
\]
The continuant recurrence is
\begin{equation}\label{eq:continuant}
P_0=1,\qquad P_1=\lambda+a_1,\qquad
P_k=(\lambda+a_k)P_{k-1}-tb_{k-1}^2P_{k-2}.
\end{equation}
Our problem is to describe the parameters for which $\chi_n$ is reducible in
$\C[\lambda,w]$.

This finite-dimensional question is motivated by work of Djakov and Mityagin
on spectral Riemann surfaces of infinite tridiagonal pencils
\cite{DjMi1,ADM}.  Several special configurations give immediate
factorizations---cuts, constant branches, reversal symmetry, and scalar
diagonal blocks---but they do not by themselves describe reducibility on the
open stratum with distinct diagonal entries and non-zero couplings.  On that
stratum it is more effective to label a factor by the subset of roots it
selects at $w=0$.

The paper has three main components.  First, we prove generic irreducibility
for a fixed simple diagonal and classify the low degrees.  In degree four
the connected pairwise-distinct reducible locus is completely explicit: in
the squared couplings $c_i=b_i^2$ it is smooth, irreducible, and of
codimension two.

Second, we prove two structural theorems valid in arbitrary degree.  For even
$n$, every polynomial factor has even $\lambda$-degree.  For odd $n$, a
linear factor can occur only at the central site, where it is characterized
by a mirror relation between the two half-chains.  We also exhibit the
endpoint family
\[
c_k=\kappa(a_{k+1}-a_1)(a_k-a_n),\qquad 1\le k\le n-1,
\]
for which
\[
P_n=
\Bigl[(\lambda+a_1)(\lambda+a_n)+\kappa(a_1-a_n)^2t\Bigr]\,
\widetilde P_{n-2}.
\]
Thus the type $2+(n-2)$ occurs in every degree on a smooth irreducible family
of codimension $n-2$.  Iteration at a distinguished fixed family gives the
spin factorization.

Third, we introduce the subset strata $\mathcal R_{n,S}$ and a bounded
coefficient-matching incidence problem.  For a factor of $\lambda$-degree
$d$ this gives the natural codimension
\[
\kappa(n,d)=\left\lceil\frac{d(n-d)}2\right\rceil,
\]
and proves that every irreducible component of a non-empty subset stratum has
codimension at most $\kappa(n,d)$.  Exact degree-eight Jacobian calculations
produce transverse branches of codimensions $6$ and $8$ for the splitting
types $2+6$ and $4+4$.  These results motivate two conjectures: that every
specialization subset is invariant under reversal of the chain, and that
non-empty subset strata have the expected codimension.  On the real Jacobi
locus we also prove an independent symmetry constraint on the ranks of the
selected diagonal entries.

The paper is organized accordingly.  Sections~\ref{sec:generic}--\ref{sec:low}
establish the general algebraic setup and the complete low-degree results.
Section~\ref{sec:central} treats linear factors, Section~\ref{sec:endpoint}
proves the peeling identity and studies the endpoint family, and
Section~\ref{sec:factor-strata} develops the subset-stratum viewpoint and the
degree-eight evidence.  Section~\ref{sec:outlook} formulates the general
conjectural picture.  The appendix records only the exact computational
certificates used in the text.

\section{Generic irreducibility}\label{sec:generic}

We first record a degree bound that will be used repeatedly.

\begin{lemma}[Degree bound for factors]\label{lem:degree-bound}
Let
\[
Q(\lambda,w)=\lambda^d+q_1(w)\lambda^{d-1}+\cdots+q_d(w)
\]
be a monic factor of $\chi_n(\lambda,w)$.  Then
\[
\deg_w q_j\le j,\qquad 1\le j\le d.
\]
If, in addition, $a_1,\ldots,a_n$ are pairwise distinct, then every monic
factor of $\chi_n$ is even in $w$.  Consequently
\[
q_j(w)=\widetilde q_j(t),\qquad
\deg_t\widetilde q_j\le \left\lfloor\frac j2\right\rfloor.
\]
\end{lemma}

\begin{proof}
For fixed $w$, every root of $Q(\cdot,w)$ is a root of $\chi_n(\cdot,w)$,
hence is the negative of an eigenvalue of $A+wB$.  Thus every such root is
$O(1+|w|)$ as $|w|\to\infty$.  The coefficient $q_j$ is, up to sign, the
$j$th elementary symmetric function of $d$ of these roots, so
$q_j(w)=O(|w|^j)$.  Since $q_j$ is a polynomial, $\deg_w q_j\le j$.

Assume now that the $a_i$ are pairwise distinct.  At $w=0$ we have
\[
\chi_n(\lambda,0)=\prod_{i=1}^n(\lambda+a_i),
\]
which is square-free.  Hence, after making $Q$ monic, there is a unique
non-empty subset $S\subset\{1,\ldots,n\}$ such that
\[
Q(\lambda,0)=\prod_{i\in S}(\lambda+a_i).
\]
Let $R=\chi_n/Q$.  The two specializations $Q(\lambda,0)$ and
$R(\lambda,0)$ are coprime.  Hensel uniqueness in $\C[[w]][\lambda]$ therefore
shows that a monic factor with this specialization is unique.  But
$Q(\lambda,-w)$ is another such factor, so $Q(\lambda,-w)=Q(\lambda,w)$.
The final degree bound follows from $t=w^2$.
\end{proof}

\begin{theorem}[Generic irreducibility]\label{thm:generic}
Assume that $a_1,\ldots,a_n$ are pairwise distinct.  For these fixed diagonal
entries, $\chi_n(\lambda,w)$ is irreducible in $\C[\lambda,w]$ for all
$(b_1,\ldots,b_{n-1})$ outside a proper Zariski-closed subset of
$\C^{n-1}$.
\end{theorem}

\begin{proof}
Fix a non-empty proper subset $S\subset\{1,\ldots,n\}$ and put
\[
F_0(\lambda)=\prod_{i\in S}(\lambda+a_i),\qquad
G_0(\lambda)=\prod_{i\notin S}(\lambda+a_i).
\]
Their resultant is a non-zero complex number.  Therefore, with the $b_i$
regarded as indeterminates, Hensel lifting gives a unique factorization
\[
\chi_n(\lambda,w)=\widehat F_S(\lambda,w)\widehat G_S(\lambda,w)
\]
in
$\C[b_1,\ldots,b_{n-1}][[w]][\lambda]$, with $\widehat F_S$ and
$\widehat G_S$ monic in $\lambda$ and with specializations $F_0,G_0$ at
$w=0$.

Let $d=|S|$.  By Lemma~\ref{lem:degree-bound}, an actual polynomial factor
with specialization $F_0$ must have the coefficient of
$\lambda^{d-j}$ of $w$-degree at most $j$.  Hence the locus
$\mathcal E_S\subset\C^{n-1}$ on which such a factor exists is the common
zero set of the higher $w$-coefficients in the formal factor
$\widehat F_S$.  These coefficients are polynomials in the $b_i$ (the Hensel
recursion divides only by the non-zero constant resultant), so
$\mathcal E_S$ is Zariski closed.  Conversely, if those higher
coefficients vanish, $\widehat F_S$ is a polynomial; monic division of
$\chi_n$ by $\widehat F_S$ then shows that it is an actual polynomial
factor.  The ideal generated by all coefficients beyond the degree bounds is
finitely generated by the Noetherian property of the polynomial ring in the
$b_i$.

It remains to see that $\mathcal E_S$ is proper.  Since the path graph is
connected, some edge $k$ has exactly one endpoint in $S$.  Specialize
$b_k=1$ and $b_i=0$ for $i\ne k$.  Then
\[
\chi_n(\lambda,w)=
\bigl((\lambda+a_k)(\lambda+a_{k+1})-w^2\bigr)
\prod_{j\ne k,k+1}(\lambda+a_j).
\]
The quadratic factor is irreducible in $\C[\lambda,w]$, because its
discriminant
\[
(a_k-a_{k+1})^2+4w^2
\]
is not a square in $\C[w]$.  Thus every factor of the specialized polynomial
contains either both of the factors $\lambda+a_k,\lambda+a_{k+1}$ at $w=0$
or neither of them.  It cannot have specialization $F_0$, so this point does
not lie in $\mathcal E_S$.

There are only finitely many non-empty proper subsets $S$.  Their proper
closed loci $\mathcal E_S$ therefore have a proper finite union.  Outside
that union no non-trivial factorization is possible.
\end{proof}

\begin{proposition}[Parity of factors in even degree]\label{prop:even-factor-degree}
Assume that $n$ is even, the diagonal entries $a_1,\ldots,a_n$ are pairwise
distinct, and $b_1\cdots b_{n-1}\ne0$.  Then every non-constant monic factor
of $\chi_n(\lambda,w)$ has even $\lambda$-degree.
\end{proposition}

\begin{proof}
Let $Q$ be a monic factor of $\lambda$-degree $d$, and write
$\chi_n=QR$.  By Lemma~\ref{lem:degree-bound}, $Q$ is even in $w$ and has
total degree at most $d$ in $(\lambda,w)$.  Let $Q_d$ denote its homogeneous
part of total degree $d$.  The same degree bound applies to $R$, and taking
top homogeneous parts in $\chi_n=QR$ gives
\[
Q_d(\lambda,w)R_{n-d}(\lambda,w)=\det(\lambda I+wB).
\]
Since $Q_d$ is homogeneous of degree $d$ and even in $w$, the polynomial
$q(\lambda)=Q_d(\lambda,1)$ satisfies
\[
q(-\lambda)=Q_d(-\lambda,1)=Q_d(-\lambda,-1)
=(-1)^dQ_d(\lambda,1)=(-1)^dq(\lambda).
\]
If $d$ were odd, then $q(0)=0$.  Specializing the displayed factorization at
$w=1$ would then imply $\det B=0$.  On the other hand, for an even
zero-diagonal tridiagonal matrix one has
\[
\det B=(-1)^{n/2}b_1^2b_3^2\cdots b_{n-1}^2\ne0,
\]
a contradiction.  Hence $d$ is even.
\end{proof}

\section{Elementary reducibility mechanisms and the principal stratum} \label{sec:principal}
Put $c_i=b_i^2$.  When some couplings vanish, we first cut the path at all
zero couplings and call the resulting maximal intervals the \emph{connected
components} of the chain.

\begin{proposition}\label{prop:elementary}
The following operations give factorizations of the spectral polynomial.
\begin{enumerate}
\item If $b_i=0$, then
\[
\chi_n=\chi_i^{(1,\ldots,i)}\,
\chi_{n-i}^{(i+1,\ldots,n)}.
\]
\item If the characteristic polynomial $\chi_I$ of a connected component
$I$ satisfies $\chi_I(-a,w)\equiv0$, then $\lambda+a$ divides $\chi_I$.
\item Suppose a connected component $I=[r,s]$ satisfies
\[
a_{r+k}=a_{s-k},\qquad b_{r+k}^2=b_{s-k-1}^2
\]
for all admissible $k$.  After a diagonal sign conjugation, the block
commutes with reversal.  Hence $\chi_I$ is the product of the characteristic
polynomials of its restrictions to the two reversal eigenspaces.
\item If a connected component $I=[r,s]$ of length at least two has scalar
diagonal part,
\[
a_r=\cdots=a_s=a,
\]
then, if $\mu_1,\ldots,\mu_{|I|}$ are the eigenvalues of the constant
off-diagonal matrix $B_I$, counted with algebraic multiplicity,
\[
\chi_I(\lambda,w)=\prod_{j=1}^{|I|}(\lambda+a+\mu_j w).
\]
\end{enumerate}
\end{proposition}

\begin{proof}
The first two assertions are immediate.  For the third, on a connected path the equalities of squared couplings allow
a diagonal conjugation by signs so that the transformed couplings are
invariant under reversal.  The transformed Jacobi block has the same
characteristic polynomial, commutes with the reversal operator, and hence
preserves its $+1$ and $-1$ eigenspaces.

For the fourth assertion,
\[
\chi_I(\lambda,w)=\det((\lambda+a)I+wB_I).
\]
The characteristic polynomial of $B_I$ splits over $\C$, and homogenizing it
gives the displayed product.  No diagonalizability assumption on $B_I$ is
needed.
\end{proof}

Put
\[
c_i=b_i^2,\qquad 1\le i\le n-1.
\]
The polynomial $P_n$ depends on the couplings only through the $c_i$.  We
therefore work, when convenient, in the squared-coupling parameter space and
set
\[
\mathcal U_n=\{(a,c)\in\C^{2n-1}:a_i\ne a_j\ (i\ne j),\
c_1\cdots c_{n-1}\ne0\}.
\]
Let $\mathcal R_n\subset\mathcal U_n$ denote the reducible locus.  Thus
$\mathcal U_n$ removes cuts and diagonal collisions while retaining the
connected reducibility phenomena that are the main subject of the paper.
The passage from $b_i$ to $c_i=b_i^2$ is a finite \'etale covering over the
connected stratum, so codimensions are unchanged.

One useful consequence of connectedness is the absence of constant branches
in even degree.  Indeed, if $n=2m$, then the coefficient of $t^m$ in
$P_{2m}$ is
\[
(-1)^m c_1c_3\cdots c_{2m-1}\ne0.
\]
Hence $P_{2m}(-a,t)$ cannot vanish identically for any constant $a$.  On
$\mathcal U_{2m}$, Proposition~\ref{prop:even-factor-degree} gives the
stronger statement that every factor has even $\lambda$-degree.

For odd $n=2m+1$, a prescribed constant branch $\lambda+a_j$ is equivalent
to
\[
P_n(-a_j,t)\equiv0.
\]
Since the constant term already vanishes, this is the system obtained by
setting to zero the $m$ non-constant coefficients of $P_n(-a_j,t)$.  We do
not assume that these equations are independent on every stratum.

\section{Reducibility in low degrees}\label{sec:low}
Throughout this section we put
\[
t=w^2,\qquad P_n(\lambda,t)=\chi_n(\lambda,w).
\]

\subsection{Degrees two and three}

\begin{proposition}\label{prop:d2}
For
\[
\chi_2=(\lambda+a_1)(\lambda+a_2)-b_1^2w^2
\]
the spectral curve is reducible if and only if either $b_1=0$ or
$a_1=a_2$.
\end{proposition}

\begin{proof}
As a monic quadratic in $\lambda$, $\chi_2$ is reducible exactly when its
discriminant
\[
(a_1-a_2)^2+4b_1^2w^2
\]
is a square in $\C[w]$.  A polynomial $\alpha+\beta w^2$ with
$\alpha\beta\ne0$ is not a square in $\C[w]$.  Hence reducibility occurs
exactly when $b_1=0$ or $a_1=a_2$.  In the latter case
\[
\chi_2=(\lambda+a_1-b_1w)(\lambda+a_1+b_1w).
\]
\end{proof}

\begin{proposition}\label{prop:d3}
Assume $b_1b_2\ne0$.  Then $\chi_3$ is reducible if and only if it has a
constant branch.  More explicitly, this happens if and only if
\[
a_1=a_3
\]
or
\[
(a_3-a_2)b_1^2+(a_1-a_2)b_2^2=0.
\]
If $b_1=0$ or $b_2=0$, the chain is disconnected and $\chi_3$ is reducible.
\end{proposition}

\begin{proof}
A non-trivial factorization of the monic cubic contains a monic
$\lambda$-degree-one factor
\[
L(\lambda,w)=\lambda+q(w).
\]
By Lemma~\ref{lem:degree-bound}, $\deg q\le1$.  Since $\chi_3$ is even in $w$,
$L(\lambda,-w)$ is also a factor.  If the two linear factors are distinct,
the remaining linear factor is invariant under $w\mapsto-w$; if they are the
same, $L$ itself is invariant.  In either case $\chi_3$ has an invariant
linear factor.  Its coefficient $q(w)$ is even and has degree at most one,
so it is constant.  Thus reducibility implies a constant branch.

The recurrence gives
\[
P_3=(\lambda+a_1)(\lambda+a_2)(\lambda+a_3)
-t\bigl(c_1(\lambda+a_3)+c_2(\lambda+a_1)\bigr).
\]
Consequently
\[
P_3(-a_1,t)=-t c_1(a_3-a_1),
\]
\[
P_3(-a_2,t)=-t\bigl(c_1(a_3-a_2)+c_2(a_1-a_2)\bigr),
\]
\[
P_3(-a_3,t)=-t c_2(a_1-a_3).
\]
Since $c_1c_2\ne0$, these identities give exactly the two alternatives in
the statement.
\end{proof}

\subsection{Degree four: a complete classification on $\mathcal U_4$}

Degree four is the first non-trivial even case on the principal stratum.
Here the whole connected pairwise-distinct reducible locus can be described
explicitly.

\begin{theorem}[Complete quartic classification]\label{thm:d4-complete}
Assume
\[
a_i\ne a_j\quad(i\ne j),\qquad c_1c_2c_3\ne0.
\]
Then $P_4(\lambda,t)$, equivalently $\chi_4(\lambda,w)$, is reducible if and
only if there exists
$\kappa\in\C^\times$ such that
\begin{equation}\label{eq:d4-param}
\begin{aligned}
c_1&=\kappa(a_1-a_4)(a_2-a_1),\\
c_2&=-\kappa(a_1-a_3)(a_2-a_4),\\
c_3&=\kappa(a_1-a_4)(a_4-a_3).
\end{aligned}
\end{equation}
In that case the factorization is
\begin{equation}\label{eq:d4-factor}
\begin{aligned}
P_4(\lambda,t)
={}&\Bigl[(\lambda+a_1)(\lambda+a_4)
+\kappa(a_1-a_4)^2t\Bigr]\\
&\times
\Bigl[(\lambda+a_2)(\lambda+a_3)
+\kappa(a_2-a_1)(a_4-a_3)t\Bigr].
\end{aligned}
\end{equation}
Thus every reducible quartic in $\mathcal U_4$ pairs the two endpoints
$\{1,4\}$ against the two middle vertices $\{2,3\}$ at $t=0$.
\end{theorem}

\begin{proof}
By Proposition~\ref{prop:even-factor-degree}, every non-constant factor has
even $\lambda$-degree.  Hence every non-trivial quartic factorization has
type $2+2$.  Lemma~\ref{lem:degree-bound} also shows that both factors are
even in $w$.  Therefore, after choosing an unordered partition
$S\sqcup T=\{1,2,3,4\}$ with $|S|=|T|=2$, one can write
\[
F=F_0+xt,\qquad G=G_0+yt,
\]
where
\[
F_0=\prod_{i\in S}(\lambda+a_i),\qquad
G_0=\prod_{i\in T}(\lambda+a_i).
\]
Comparing with
\[
\begin{aligned}
P_4={}&\prod_{i=1}^4(\lambda+a_i)\\
&-t\!\left[
c_1(\lambda+a_3)(\lambda+a_4)
+c_2(\lambda+a_1)(\lambda+a_4)
+c_3(\lambda+a_1)(\lambda+a_2)\right]
+t^2c_1c_3.
\end{aligned}
\]
gives a linear system for $c_1,c_2,c_3$ from the coefficient of $t$, together
with
\begin{equation}\label{eq:d4-xy}
xy=c_1c_3
\end{equation}
from the coefficient of $t^2$.

There are three unordered partitions.  For
$S=\{1,2\}$, $T=\{3,4\}$, the linear equations give
\[
c_1=-x,\qquad c_2=0,\qquad c_3=-y,
\]
which is impossible on the connected stratum.

For $S=\{1,3\}$, $T=\{2,4\}$, they give
\[
c_1=-x\,\frac{a_1-a_2}{a_1-a_3},\qquad
c_3=-y\,\frac{a_3-a_4}{a_2-a_4}.
\]
Substitution into \eqref{eq:d4-xy} yields
\[
xy\,\frac{(a_1-a_4)(a_2-a_3)}
{(a_1-a_3)(a_2-a_4)}=0.
\]
All four diagonal differences are non-zero, and connectedness forces
$x,y\ne0$, so this partition is impossible as well.

It remains to take $S=\{1,4\}$ and $T=\{2,3\}$.  The coefficient of $t$
gives
\[
c_1=x\,\frac{a_2-a_1}{a_1-a_4},\qquad
c_3=x\,\frac{a_4-a_3}{a_1-a_4},
\]
and
\[
c_2=-\frac{a_1y+a_2x-a_3x-a_4y}{a_1-a_4}.
\]
Equation~\eqref{eq:d4-xy} now gives
\[
y=x\,\frac{(a_2-a_1)(a_4-a_3)}{(a_1-a_4)^2}.
\]
Putting $\kappa=x/(a_1-a_4)^2$ gives exactly
\eqref{eq:d4-param} and \eqref{eq:d4-factor}.  Conversely, direct
multiplication of \eqref{eq:d4-factor} gives the continuant with the
couplings \eqref{eq:d4-param}, so the conditions are sufficient.
\end{proof}

\begin{corollary}[Geometry of the quartic reducible locus]
\label{cor:d4-geometry}
The locus $\mathcal R_4$ is smooth and irreducible of dimension $5$, hence of
codimension $2$ in $\mathcal U_4$.
\end{corollary}

\begin{proof}
Theorem~\ref{thm:d4-complete} identifies $\mathcal R_4$ with the image of
\[
\{(a_1,a_2,a_3,a_4,\kappa):
a_i\ne a_j,\ \kappa\ne0\}
\]
under the map \eqref{eq:d4-param}.  The $a_i$ are unchanged, and
\[
\kappa=\frac{c_1}{(a_1-a_4)(a_2-a_1)},
\]
so the map is an isomorphism onto its image.  The parameter space on the
left is a smooth irreducible open subset of $\C^5$.
\end{proof}

\begin{remark}[Boundary strata]
Theorem~\ref{thm:d4-complete} concerns the principal stratum $\mathcal U_4$.
On its boundary, zero couplings and collisions among the diagonal entries
produce additional reducible families, including the reversal-symmetric and
scalar-diagonal families of Proposition~\ref{prop:elementary}.  Thus the
theorem is a complete classification of the connected pairwise-distinct
quartic locus, not a classification of every collision stratum.
\end{remark}

\section{Centrality of constant branches}\label{sec:central}

For odd $n=2m+1$ the smallest value of the expected codimension
$\kappa(n,d)$ of Section~\ref{sec:factor-strata} occurs at $d=1$, so
degree-one factors control the minimal codimension.  In this section we
determine completely which sites of the chain can carry such a factor: on the
principal stratum, only the central one can.

For $0\le k\le n+1$ write
\[
L_k=\chi^{(1,\ldots,k)},\qquad R_k=\chi^{(k,\ldots,n)}
\]
for the spectral polynomials of the initial and final sub-chains, with the
conventions
\[
L_0=R_{n+1}=1,\qquad L_{-1}=R_{n+2}=0,\qquad c_0=c_n=0 .
\]

\begin{lemma}[Vertex cut]\label{lem:vertex-cut}
For every $1\le j\le n$,
\begin{equation}\label{eq:vertex-cut}
P_n=(\lambda+a_j)L_{j-1}R_{j+1}
   -tc_{j-1}L_{j-2}R_{j+1}-tc_jL_{j-1}R_{j+2}.
\end{equation}
Consequently $\lambda+a_j$ divides $P_n$ in $\C[\lambda,t]$ if and only if
\begin{equation}\label{eq:const-branch}
\bigl(c_{j-1}L_{j-2}R_{j+1}+c_jL_{j-1}R_{j+2}\bigr)\Big|_{\lambda=-a_j}
\equiv0\quad\text{in }\C[t].
\end{equation}
\end{lemma}

\begin{proof}
Expand the determinant along the $j$th row and then the corresponding
minors.  The three contributions are
\[
(\lambda+a_j)L_{j-1}R_{j+1},\qquad
-w^2b_{j-1}^2L_{j-2}R_{j+1},\qquad
-w^2b_j^2L_{j-1}R_{j+2},
\]
which gives \eqref{eq:vertex-cut}.  Setting $\lambda=-a_j$ proves
\eqref{eq:const-branch}.
\end{proof}

\begin{lemma}[Coprimality of consecutive sub-chain polynomials]
\label{lem:coprime}
Let $(a,c)\in\mathcal U_n$, fix $j$, and specialize $\lambda=-a_j$.  Then, in
$\C[t]$,
\[
\gcd\bigl(L_{j-1},L_{j-2}\bigr)=1,\qquad
\gcd\bigl(R_{j+1},R_{j+2}\bigr)=1 .
\]
\end{lemma}

\begin{proof}
The left identity is immediate for $j\le2$.  Assume $j\ge3$ and put
$g=\gcd(L_{j-1},L_{j-2})$.  Since
$L_k(-a_j,0)=\prod_{i\le k}(a_i-a_j)\ne0$ for $k<j$, we have $g(0)\ne0$,
so $g$ is coprime to $t$.  From
\[
tc_{j-2}L_{j-3}=(\lambda+a_{j-1})L_{j-2}-L_{j-1}
\]
and $c_{j-2}\ne0$ it follows that $g\mid L_{j-3}$.  Repeating the argument
gives $g\mid L_0=1$.  The right identity follows by reversing the chain.
\end{proof}

\begin{theorem}[Mirror criterion and centrality]\label{thm:central}
Let $n\ge2$, $(a,c)\in\mathcal U_n$, and $1\le j\le n$.
\begin{enumerate}
\item The polynomial $\lambda+a_j$ divides $\chi_n$ if and only if
$2\le j\le n-1$ and there exist $\sigma,\tau\in\C^\times$ such that, as
identities in $\C[t]$ after setting $\lambda=-a_j$,
\begin{equation}\label{eq:mirror}
L_{j-1}=\sigma R_{j+1},\qquad
L_{j-2}=\tau R_{j+2},\qquad
c_{j-1}\tau+c_j\sigma=0 .
\end{equation}
\item If $\chi_n$ has a factor of $\lambda$-degree one, then $j-1=n-j$.  In
particular $n$ is odd, and the only possible degree-one factor of $\chi_n$ on
$\mathcal U_n$ is
\[
\lambda+a_{(n+1)/2}.
\]
\end{enumerate}
\end{theorem}

\begin{proof}
(1) Assume \eqref{eq:const-branch}.  First, $j\ne1$: for $j=1$ the
conventions $L_{-1}=0$, $L_0=1$, $c_0=0$ reduce \eqref{eq:const-branch} to
$c_1R_3(-a_1,t)\equiv0$, impossible since $c_1\ne0$ and
$R_3(-a_1,0)=\prod_{i\ge3}(a_i-a_1)\ne0$.  The case $j=n$ is symmetric.  So
$2\le j\le n-1$, and each of $L_{j-1},L_{j-2},R_{j+1},R_{j+2}$ is a product
of non-zero differences at $t=0$, hence non-zero in $\C[t]$.

Rewrite \eqref{eq:const-branch} as
\[
c_{j-1}L_{j-2}R_{j+1}=-c_jL_{j-1}R_{j+2}.
\]
The polynomial $L_{j-1}$ divides the left-hand side and is coprime to
$L_{j-2}$ by Lemma~\ref{lem:coprime}, so $L_{j-1}\mid R_{j+1}$.
Symmetrically $R_{j+1}$ divides the right-hand side and is coprime to
$R_{j+2}$, so $R_{j+1}\mid L_{j-1}$.  Hence $L_{j-1}=\sigma R_{j+1}$ with
$\sigma\in\C^\times$.  The same argument applied to the pair
$(L_{j-2},R_{j+2})$ gives $L_{j-2}=\tau R_{j+2}$ with $\tau\in\C^\times$.
Substituting and cancelling the non-zero factor $R_{j+1}R_{j+2}$ leaves
$c_{j-1}\tau+c_j\sigma=0$.  The converse is immediate.

(2) Put $\ell=j-1$ and $r=n-j$, the lengths of the two sub-chains, so that
$\ell+r=n-1$.  We use the following consequence of \eqref{eq:continuant}: the
spectral polynomial of a connected chain of length $k$ is
\[
\sum_{M}(-t)^{|M|}\Bigl(\prod_{e\in M}c_e\Bigr)
\prod_{v\notin V(M)}(\lambda+a_v),
\]
the sum being over matchings $M$ of the path on $k$ vertices.  If $k$ is
even the path has a unique perfect matching, so the coefficient of $t^{k/2}$
is $\pm c_1c_3\cdots c_{k-1}\ne0$ and is independent of $\lambda$; hence
after any specialization of $\lambda$ the $t$-degree is exactly $k/2$.  If
$k$ is odd, every maximum matching leaves one vertex uncovered, so the
coefficient of $t^{(k-1)/2}$ has $\lambda$-degree one and the $t$-degree
after specializing $\lambda$ is at most $(k-1)/2$.

If $\ell$ and $r$ are both even, then $\deg_tL_{j-1}=\ell/2$ and
$\deg_tR_{j+1}=r/2$ exactly, so the first relation in \eqref{eq:mirror}
forces $\ell=r$.  If $\ell$ and $r$ are both odd, then $\ell-1$ and $r-1$
are even, so $\deg_tL_{j-2}=(\ell-1)/2$ and $\deg_tR_{j+2}=(r-1)/2$ exactly,
and the second relation in \eqref{eq:mirror} forces $\ell=r$.  If $\ell$ and
$r$ have opposite parities then $n=\ell+r+1$ is even, and $\chi_n$ has no
degree-one factor at all: the coefficient of $t^{n/2}$ in $P_n$ is
$(-1)^{n/2}c_1c_3\cdots c_{n-1}\ne0$ and independent of $\lambda$, so
$P_n(-a,t)\not\equiv0$ for every constant $a$.

Thus $\ell=r$, i.e. $n$ is odd and $j=(n+1)/2$.
\end{proof}

\begin{remark}
For $n=3$ Theorem~\ref{thm:central} recovers the second alternative of
Proposition~\ref{prop:d3}: the only admissible site is $j=2$, and
\eqref{eq:mirror} reads $L_1=\sigma R_3$, $L_0=\tau R_4$, i.e.
$a_1-a_2=\sigma(a_3-a_2)$ and $1=\tau$, together with
$c_1+c_2\sigma=0$; eliminating $\sigma$ gives
$c_1(a_3-a_2)+c_2(a_1-a_2)=0$.  For $n=5$ the theorem says that only $j=3$
is possible, and an exact Gr\"obner computation confirms that
$\mathcal R_{5,\{j\}}=\emptyset$ for $j\ne3$ while
$\mathcal R_{5,\{3\}}\ne\emptyset$.
\end{remark}

\begin{corollary}[The constant-branch stratum]\label{cor:const-stratum}
Let $n=2m+1$ with $m\ge1$ and let $j=m+1$.  On $\mathcal U_n$ the stratum
$\mathcal R_{n,\{j\}}$ is the zero locus of the coefficients of
\[
c_mL_{m-1}R_{m+2}+c_{m+1}L_mR_{m+3}\Big|_{\lambda=-a_{m+1}}\in\C[t],
\]
a polynomial of $t$-degree at most $m-1$; this is a system of $m$ equations.
Equivalently, in the form \eqref{eq:mirror}, normalizing $\sigma$ and $\tau$
by the constant terms leaves
\[
\Bigl\lfloor\tfrac m2\Bigr\rfloor+\Bigl\lfloor\tfrac{m-1}2\Bigr\rfloor+1=m
\]
equations.  Both presentations therefore consist of exactly
$\kappa(n,1)=m$ equations, in agreement with
Proposition~\ref{prop:kappa}.
\end{corollary}

Condition \eqref{eq:mirror} says that, at
$\lambda=-a_{m+1}$, the two half-chains have proportional spectral
polynomials at two consecutive truncation levels.

\section{The endpoint family and the peeling identity}\label{sec:endpoint}

The quartic family of Theorem~\ref{thm:d4-complete} extends to every
degree.  Its factorization is governed by an exact \emph{peeling identity}:
a quadratic factor splits off and the complementary factor is the spectral
polynomial of the chain obtained by deleting the two end sites.  The resulting
family realizes the expected codimension for the splitting $2+(n-2)$.

\subsection{The peeling identity}

Fix $n\ge3$ and free parameters $a_1,\ldots,a_n,\kappa$.  For $1\le m\le n$
introduce two chains, in both of which the displayed coupling is the one
joining the sites carrying $a_k$ and $a_{k+1}$:
\[
\mathcal D_m:\quad\text{diagonal }(a_1,\ldots,a_m),\qquad
c_k=\kappa(a_{k+1}-a_1)(a_k-a_n),
\]
\[
\mathcal E_m:\quad\text{diagonal }(a_2,\ldots,a_m),\qquad
\varepsilon_k=\kappa(a_k-a_1)(a_{k+1}-a_n),\quad 2\le k\le m-1.
\]
Thus $\mathcal E_m$ is obtained from $\mathcal D_m$ by deleting the first site
and interchanging the roles of the two reference values $a_1$ and $a_n$ in the
couplings.  Let $D_m$ and $E_m$ be the corresponding spectral polynomials,
with the conventions $D_0=1$, $E_1=1$, $E_0=0$.

\begin{lemma}[Two-term relation]\label{lem:two-term}
For every $m\ge1$,
\begin{equation}\label{eq:two-term}
D_m=(\lambda+a_1)E_m+\kappa(a_1-a_m)(a_1-a_n)\,t\,E_{m-1}.
\end{equation}
\end{lemma}

\begin{proof}
The cases $m=1,2$ are immediate.  For $m\ge3$ apply the continuant
recurrence to $D_m,E_m,E_{m-1}$ and use the induction hypotheses for
$D_{m-1}$ and $D_{m-2}$.  Put
\[
A=a_1,\qquad x=a_{m-1},\qquad y=a_m,\qquad z=a_n.
\]
After eliminating $E_{m-1}$, the coefficient of $E_{m-3}$ cancels because
$(A-y)(a_{m-2}-A)=(y-A)(A-a_{m-2})$.  The remaining coefficient is
$\kappa t$ times
\[
\begin{aligned}
&(A-z)\bigl[(\lambda+y)(A-x)-(A-y)(\lambda+x)\bigr]\\
&\quad+(\lambda+A)\bigl[(x-A)(y-z)-(y-A)(x-z)\bigr].
\end{aligned}
\]
The two brackets are $(y-x)(\lambda+A)$ and $(y-x)(z-A)$, respectively,
so the expression vanishes.
\end{proof}

\begin{theorem}[Peeling identity]\label{thm:peel}
Let $n\ge3$, let $a_1,\ldots,a_n$ and $\kappa$ be arbitrary, and put
\begin{equation}\label{eq:endpoint-c}
c_k=\kappa\,(a_{k+1}-a_1)(a_k-a_n),\qquad 1\le k\le n-1 .
\end{equation}
Then
\begin{equation}\label{eq:peel}
P_n(\lambda,t)=
\Bigl[(\lambda+a_1)(\lambda+a_n)+\kappa(a_1-a_n)^2\,t\Bigr]\,
\widetilde P_{n-2}(\lambda,t),
\end{equation}
where $\widetilde P_{n-2}$ is the spectral polynomial of the chain with
diagonal $(a_2,\ldots,a_{n-1})$ and squared couplings
\begin{equation}\label{eq:peeled-c}
\widetilde c_k=\kappa\,(a_{k+1}-a_1)(a_{k+2}-a_n),\qquad 1\le k\le n-3 .
\end{equation}
\end{theorem}

\begin{proof}
Take $m=n$ in Lemma~\ref{lem:two-term}.  The last coupling of
$\mathcal E_n$ is $\varepsilon_{n-1}=\kappa(a_{n-1}-a_1)(a_n-a_n)=0$, so
that chain is disconnected at its last edge and $E_n=(\lambda+a_n)E_{n-1}$.
Hence
\[
\begin{aligned}
D_n
&=(\lambda+a_1)(\lambda+a_n)E_{n-1}
  +\kappa(a_1-a_n)^2tE_{n-1}\\
&=\Bigl[(\lambda+a_1)(\lambda+a_n)
  +\kappa(a_1-a_n)^2t\Bigr]E_{n-1}.
\end{aligned}
\]
By \eqref{eq:endpoint-c} we have $D_n=P_n$, and $E_{n-1}$ is precisely the
chain with diagonal $(a_2,\ldots,a_{n-1})$ and couplings
\eqref{eq:peeled-c}.
\end{proof}

\subsection{Geometry of the endpoint family}

\begin{corollary}[An explicit reducible family of expected codimension]
\label{cor:endpoint-geometry}
For $n\ge3$ put
\[
\mathcal N_n=\Bigl\{(a,c)\in\mathcal U_n:\
c_k=\kappa(a_{k+1}-a_1)(a_k-a_n)\ (1\le k\le n-1)
\text{ for some }\kappa\in\C^\times\Bigr\}.
\]
Then $\mathcal N_n\subset\mathcal R_{n,\{1,n\}}$, and $\mathcal N_n$ is a
smooth irreducible locally closed subvariety of $\mathcal U_n$ isomorphic to
an open subset of $\C^{n+1}$.  In particular
\[
\dim\mathcal N_n=n+1,\qquad
\operatorname{codim}_{\mathcal U_n}\mathcal N_n=n-2=\kappa(n,2).
\]
\end{corollary}

\begin{proof}
The set of $(a_1,\ldots,a_n,\kappa)$ with the $a_i$ pairwise distinct and
$\kappa\ne0$ is a smooth irreducible open subset of $\C^{n+1}$, and
\eqref{eq:endpoint-c} gives $c_k\ne0$ for all $k$, because $a_{k+1}\ne a_1$
for $k\ge1$ and $a_k\ne a_n$ for $k\le n-1$; so the image lies in
$\mathcal U_n$.  The map leaves the $a_i$ unchanged and satisfies
$\kappa=c_1/\bigl((a_2-a_1)(a_1-a_n)\bigr)$, hence is an isomorphism onto its
image.  Finally $\dim\mathcal U_n=2n-1$ and
$\kappa(n,2)=\lceil2(n-2)/2\rceil=n-2$.
\end{proof}

\begin{corollary}[Global upper bound in every degree]\label{cor:endpoint-upper}
For every $n\ge3$ the stratum $\mathcal R_{n,\{1,n\}}$ is non-empty and
\[
\operatorname{codim}_{\mathcal U_n}\mathcal R_n\le n-2 .
\]
\end{corollary}

For even $n$ this recovers Corollary~\ref{cor:global-upper} without any
appeal to the spin family; for odd $n$ the bound $\kappa(n,1)=(n-1)/2$ of
Corollary~\ref{cor:global-upper}, coming from the central constant branch of
Theorem~\ref{thm:central}, is stronger.

Thus $\mathcal N_n$ already has the expected codimension for the splitting
$2+(n-2)$.  Its closure in $\mathcal U_n$ is an irreducible subvariety of
$\mathcal R_{n,\{1,n\}}$ of codimension $n-2$; in particular,
Conjecture~\ref{con:expected-codim} would imply that this closure is an
irreducible component of the endpoint stratum.

\begin{corollary}[Low-degree exactness]\label{cor:endpoint-low}
On the principal stratum,
\[
\mathcal N_3=\mathcal R_3,\qquad
\mathcal N_4=\mathcal R_4.
\]
\end{corollary}

\begin{proof}
For $n=4$ this is Theorem~\ref{thm:d4-complete}, since
\eqref{eq:endpoint-c} is exactly \eqref{eq:d4-param}.  For $n=3$,
Proposition~\ref{prop:d3} gives
\[
c_1(a_3-a_2)+c_2(a_1-a_2)=0.
\]
Solving for $c_2$ and setting
$\kappa=c_1/((a_2-a_1)(a_1-a_3))$ gives precisely
\eqref{eq:endpoint-c}.
\end{proof}

\subsection{The spin family}

\begin{corollary}[Spin factorization]\label{cor:spin-family}
For $n\ge2$ and
\[
a_k=n+1-2k,\qquad c_k=k(n-k),
\]
one has, writing $n=2m$ or $n=2m+1$,
\[
P_{2m}(\lambda,t)=
\prod_{r=1}^{m}\bigl[\lambda^2-(2r-1)^2(1+t)\bigr],
\]
\[
P_{2m+1}(\lambda,t)=
\lambda\prod_{r=1}^{m}\bigl[\lambda^2-(2r)^2(1+t)\bigr].
\]
\end{corollary}

\begin{proof}
The case $n=2$ is immediate.  For $n\ge3$, take
$\kappa=-\tfrac14$ in \eqref{eq:endpoint-c}.  Since
$a_{k+1}-a_1=-2k$ and $a_k-a_n=2(n-k)$, the couplings become
$c_k=k(n-k)$.  The quadratic factor in \eqref{eq:peel} is
\[
\lambda^2-(n-1)^2(1+t),
\]
and the peeled chain has the same form with $n$ replaced by $n-2$.
Iterating Theorem~\ref{thm:peel} gives the two product formulas.
\end{proof}

Thus the spin family is stable under peeling (up to lowering the length by
two).  A general member of $\mathcal N_n$ need not have this property, so its
complementary factor is typically not subject to a second peeling step.

\subsection{The real Jacobi locus and a rank symmetry}\label{sec:real}

The pencil $J_n(w)$ is a genuine Jacobi matrix for real $w$ exactly when the
$a_i$ are real and the $c_i=b_i^2$ are positive.  This locus carries extra
information, because a real symmetric tridiagonal matrix with non-vanishing
off-diagonal entries has simple spectrum, so no two eigenvalue branches ever
cross.

\begin{proposition}[Rank symmetry]\label{prop:rank-symmetry}
Let $a_1,\ldots,a_n\in\R$ be pairwise distinct, let $c_i>0$ for all $i$, and
suppose $\chi_n=QR$ with $Q$ monic in $\lambda$ of degree $d$,
$1\le d\le n-1$.  Let $S\subset\{1,\ldots,n\}$, $|S|=d$, be the subset with
$Q(\lambda,0)=\prod_{i\in S}(\lambda+a_i)$, and let $\pi$ be the permutation
with $a_{\pi(1)}>\cdots>a_{\pi(n)}$.  Then the rank set
$\pi^{-1}(S)$ is invariant under $r\mapsto n+1-r$.
\end{proposition}

\begin{proof}
For $w>0$ the matrix $A+wB$ is real symmetric tridiagonal with non-zero
off-diagonal entries, hence has $n$ simple real eigenvalues
$\mu_1(w)>\cdots>\mu_n(w)$.  The $d$ roots of $Q(\cdot,w)$ are among the $n$
simple roots $-\mu_r(w)$ of $\chi_n(\cdot,w)$, so the set of indices $r$
occurring is locally constant, hence constant, on the connected interval
$(0,\infty)$; call it $T$.  As $w\to0^+$ we have $\mu_r(0)=a_{\pi(r)}$, so
$T=\pi^{-1}(S)$.

By Lemma~\ref{lem:degree-bound}, $Q$ is even in $w$ and has total degree at
most $d$ in $(\lambda,w)$.  Let $Q_d$ be its homogeneous part of total degree
$d$; taking top homogeneous parts in $\chi_n=QR$ gives
$Q_dR_{n-d}=\det(\lambda I+wB)$.  Hence $q(\lambda):=Q_d(\lambda,1)$ divides
$\det(\lambda I+B)=\prod_r(\lambda+\nu_r)$, where $\nu_1>\cdots>\nu_n$ are
the eigenvalues of $B$, which are real and simple for the same reason.  Since
$Q_d$ is homogeneous of degree $d$ and even in $w$,
\[
q(-\lambda)=Q_d(-\lambda,1)=Q_d(-\lambda,-1)=(-1)^dQ_d(\lambda,1)
=(-1)^dq(\lambda),
\]
so the set of roots of $q$ is invariant under $\lambda\mapsto-\lambda$.
Conjugating $B$ by $\diag\bigl((-1)^k\bigr)$ sends $B$ to $-B$, so
$\operatorname{spec}B$ is symmetric, and being simple it satisfies
$\nu_{n+1-r}=-\nu_r$.  Therefore the index set
$T'=\{r:q(-\nu_r)=0\}$ is invariant under $r\mapsto n+1-r$.

Finally, $Q=Q_d+(\text{lower total degree})$ shows that as $w\to+\infty$ the
roots of $Q(\cdot,w)$ are asymptotic to $-w\nu_r$ with $r\in T'$, while
$\mu_r(w)\sim w\nu_r$; hence $T=T'$.
\end{proof}

\begin{corollary}\label{cor:real-endpoint}
The family $\mathcal N_n$ meets the real Jacobi locus exactly over those real
diagonals for which the $n-1$ numbers $(a_{k+1}-a_1)(a_k-a_n)$,
$1\le k\le n-1$, all have the same sign, with $\kappa$ chosen with that same sign.  For every such diagonal the rank set of $\{1,n\}$ is a pair
$\{r,n+1-r\}$, as it must be by Proposition~\ref{prop:rank-symmetry}.
\end{corollary}

Proposition~\ref{prop:rank-symmetry} concerns the ranks of the diagonal
entries, whereas Conjecture~\ref{con:reversal} below concerns their positions
in the chain.  The two constraints are therefore independent; on the real
Jacobi locus they must hold simultaneously.

\section{Factorization strata and degree eight}\label{sec:factor-strata}

The quartic theorem suggests organizing reducibility by the factor selected
at $t=0$.  By Lemma~\ref{lem:degree-bound}, on $\mathcal U_n$ every factor of
$\chi_n$ is even in $w$; hence reducibility of $\chi_n$ in $\C[\lambda,w]$
is equivalent there to reducibility of $P_n$ in $\C[\lambda,t]$.  The
specialization at $t=0$ therefore labels factorization strata canonically.

\subsection{Subset strata and the incidence count}

Let $S\subset\{1,\ldots,n\}$ be non-empty and proper, and put $d=|S|$.
We denote by $\mathcal R_{n,S}$ the set of points of $\mathcal U_n$ for
which $P_n$ has a monic factor $F$ satisfying
\[
F(\lambda,0)=\prod_{i\in S}(\lambda+a_i).
\]
Replacing $S$ by its complement gives the same factorization, so we usually
take $d\le n/2$.  Because $P_n(\lambda,0)$ is square-free, the specialization
subset $S$ uniquely identifies the formal Hensel factor.  When $n$ is even,
Proposition~\ref{prop:even-factor-degree} shows that only even values of $d$
can occur.

\begin{proposition}\label{prop:subset-closed}
For every $S$, the subset stratum $\mathcal R_{n,S}$ is Zariski closed in
$\mathcal U_n$.  Consequently $\mathcal R_n$ is Zariski closed in
$\mathcal U_n$.
\end{proposition}

\begin{proof}
Over $\mathcal U_n$ the two polynomials
\[
F_{0,S}(\lambda)=\prod_{i\in S}(\lambda+a_i),\qquad
G_{0,S}(\lambda)=\prod_{i\notin S}(\lambda+a_i)
\]
are coprime, and their resultant is a unit in the coordinate ring of
$\mathcal U_n$.  Hensel lifting therefore gives a unique formal factor
$\widehat F_S(\lambda,t)$ with specialization $F_{0,S}$.  Its coefficients
are regular functions on $\mathcal U_n$: the recursive construction divides
only by the above resultant.  By Lemma~\ref{lem:degree-bound}, an actual
polynomial factor of $\lambda$-degree $d=|S|$ has the coefficient of
$\lambda^{d-j}$ of $t$-degree at most $\lfloor j/2\rfloor$.  Thus
$\mathcal R_{n,S}$ is the common zero set of all higher $t$-coefficients of
$\widehat F_S$.  Conversely, if all these higher coefficients vanish,
$\widehat F_S$ is a polynomial; monic division of $P_n$ by it shows that it
is an actual factor in $\C[\lambda,t]$.  The coordinate ring of
$\mathcal U_n$ is Noetherian, so finitely many of the equations suffice.
Hence $\mathcal R_{n,S}$ is closed.  The reducible locus is a finite union of
such strata.
\end{proof}

For a monic factor of $\lambda$-degree $d$, Lemma~\ref{lem:degree-bound}
allows
\[
\nu(d):=\sum_{j=1}^d\left(\left\lfloor\frac j2\right\rfloor+1\right)
\]
bounded coefficients.  Explicitly,
\begin{equation}\label{eq:nu}
\nu(2r)=r(r+2),\qquad
\nu(2r+1)=r^2+3r+1.
\end{equation}

\begin{proposition}[Incidence count]\label{prop:kappa}
For a $d+(n-d)$ factorization on $\mathcal U_n$, the bounded
coefficient-matching incidence problem has
\[
(2n-1)+\nu(d)+\nu(n-d)
\]
variables and $\nu(n)$ coefficient equations.  Put
\begin{equation}\label{eq:kappa}
\kappa(n,d):=\nu(n)-\nu(d)-\nu(n-d)
=\left\lceil\frac{d(n-d)}2\right\rceil.
\end{equation}
Every irreducible component of a non-empty subset stratum
$\mathcal R_{n,S}$ has codimension at most $\kappa(n,d)$.  At a point of
its coefficient-matching incidence variety where the coefficient equations
have full row rank, the corresponding local branch of $\mathcal R_{n,S}$
has codimension exactly $\kappa(n,d)$.
\end{proposition}

\begin{proof}
The chain contributes $2n-1$ variables $(a_1,\ldots,a_n,c_1,\ldots,c_{n-1})$.
The two monic factors contribute $\nu(d)$ and $\nu(n-d)$ coefficients.
Matching the coefficient of every monomial
$\lambda^{n-j}t^k$, with
$1\le j\le n$ and $0\le k\le\lfloor j/2\rfloor$, gives $\nu(n)$ equations.
By the height theorem, every non-empty irreducible component of the incidence
variety has dimension at least
\[
(2n-1)+\nu(d)+\nu(n-d)-\nu(n).
\]
Over $\mathcal U_n$ the incidence variety splits into pieces according to
the subset selected at $t=0$.  On the piece labelled by $S$, Hensel
uniqueness shows that the monic factor with specialization
$\prod_{i\in S}(\lambda+a_i)$ is unique.  Equivalently, this piece is the
graph over $\mathcal R_{n,S}$ of the bounded coefficients of the regular
Hensel factor from Proposition~\ref{prop:subset-closed}.  Hence the
projection to $\mathcal R_{n,S}$ preserves dimension, giving the asserted
upper bound on codimension.  If the coefficient Jacobian has full row rank,
the incidence variety is smooth of the displayed expected dimension at that
point, and equality follows for the corresponding local branch.  The
identity \eqref{eq:kappa} follows from \eqref{eq:nu} by separating the
parities of $n$ and $d$.
\end{proof}

The formula is exact in degree four:
Theorem~\ref{thm:d4-complete} gives
$\operatorname{codim}\mathcal R_4=\kappa(4,2)=2$.
Thus the quartic classification provides the first global instance in which
the incidence count is attained.

\subsection{Which subsets occur}\label{sec:which-subsets}

Proposition~\ref{prop:kappa} bounds the codimension of a non-empty subset
stratum but does not determine which subsets occur.  Let
\[
\iota(i)=n+1-i
\]
be reversal of the chain.  Reversing the Jacobi matrix carries
$\mathcal R_{n,S}$ isomorphically onto $\mathcal R_{n,\iota S}$.

The proved low-degree results are more restrictive.  By
Theorem~\ref{thm:central}, a degree-one factor can select only the central
site, and by Theorem~\ref{thm:d4-complete} the only degree-two subset in
degree four is $\{1,4\}$ (equivalently its complement $\{2,3\}$).  In both
cases the subset is $\iota$-invariant.  The endpoint and spin families are
compatible with the same pattern, which motivates the following conjecture.

\begin{conjecture}[Reversal invariance]\label{con:reversal}
Let $(a,c)\in\mathcal U_n$ and let $S$ label a monic factor of $\chi_n$, so
that $\mathcal R_{n,S}\ne\emptyset$.  Then $S=\iota S$, that is,
$i\in S$ if and only if $n+1-i\in S$.
\end{conjecture}

Conjecture~\ref{con:reversal} is consistent with all the explicit families of
this paper: the endpoint family of Theorem~\ref{thm:peel} realizes
$S=\{1,n\}$, the central constant branch of Theorem~\ref{thm:central}
realizes $S=\{(n+1)/2\}$, and the spin family realizes
$S=\{i,n+1-i\}$ for every $i$.  It is logically independent of
Proposition~\ref{prop:rank-symmetry}, which constrains the ranks of the
diagonal entries rather than the positions of the sites, and holds only on the
real Jacobi locus; on that locus the two statements together force $S$ to be
invariant under both involutions.

\begin{corollary}[A general upper bound]\label{cor:global-upper}
For $n\ge4$,
\[
\operatorname{codim}_{\mathcal U_n}\mathcal R_n\le
\begin{cases}
\dfrac{n-1}{2},& n\ \text{odd},\\[2mm]
n-2,& n\ \text{even}.
\end{cases}
\]
Here codimension is taken inside $\mathcal U_n$.
\end{corollary}

\begin{proof}
The spin point belongs to $\mathcal U_n$ for every $n$.  If $n=2m+1$, the
factor $\lambda$ gives a non-empty degree-one incidence stratum, so
Proposition~\ref{prop:kappa} gives codimension at most
$\kappa(n,1)=m$.  If $n=2m\ge4$, the factor
$\lambda^2-(1+t)$ gives a non-empty degree-two incidence stratum, and the
same proposition gives codimension at most $\kappa(n,2)=n-2$.
\end{proof}

\subsection{Degree eight}\label{sec:degree-eight}
On $\mathcal U_8$, Proposition~\ref{prop:even-factor-degree} rules out odd
factor degrees.  Consequently every reducible polynomial admits, up to
interchanging the two factors, a binary splitting of type $2+6$ or $4+4$.
There are therefore only
\[
\binom82=28,\qquad \frac12\binom84=35
\]
specialization subsets to consider, for a total of $63$ subset strata.
Proposition~\ref{prop:kappa} predicts
\begin{equation}\label{eq:d8-profile}
\begin{array}{c|cc}
\text{splitting type}&2+6&4+4\\
\hline
\kappa(8,d)&6&8.
\end{array}
\end{equation}

At $n=8$, Corollary~\ref{cor:spin-family} gives the point
\begin{equation}\label{eq:d8-spin-point}
(a_1,\ldots,a_8)=(7,5,3,1,-1,-3,-5,-7),
\end{equation}
\[
(c_1,\ldots,c_7)=(7,12,15,16,15,12,7),
\]
with
\begin{equation}\label{eq:d8-spin-factor}
P_8(\lambda,t)=
\prod_{q\in\{1,3,5,7\}}\bigl[\lambda^2-q^2(1+t)\bigr].
\end{equation}
This point lies in $\mathcal U_8$ and belongs simultaneously to several
$2+6$ and $4+4$ subset strata.

\begin{proposition}[Exact degree-eight incidence ranks]\label{prop:d8-ranks}
At the point \eqref{eq:d8-spin-point}, fix any one of the four quadratic
factors in \eqref{eq:d8-spin-factor} and group it against the other three.
The corresponding $2+6$ bounded-coefficient incidence Jacobian has full row
rank $24$.  Likewise, for each of the three unordered groupings of the four
quadratics into two pairs, the corresponding $4+4$ incidence Jacobian has
full row rank $24$.

Consequently these incidence branches have image codimensions $6$ and $8$,
respectively, in $\mathcal U_8$.
\end{proposition}

\begin{proof}
The variable and equation counts are
\[
\begin{array}{c|c|c}
\text{splitting}&\text{variables}&\text{coefficient equations}\\
\hline
2+6&33&24\\
4+4&31&24.
\end{array}
\]
Exact rational row reduction gives rank $24$ in every grouping stated above.
Since the projection from the incidence variety to the $15$ chain parameters
has finite fibres, Proposition~\ref{prop:kappa} gives image codimensions
$15-(33-24)=6$ and $15-(31-24)=8$.  Non-zero maximal minors giving exact
rank certificates are recorded in Appendix~\ref{app:exact-checks}.
\end{proof}

\begin{conjecture}[Degree-eight classification by codimension]
\label{con:d8-class}
For every subset $S\subset\{1,\ldots,8\}$ with $|S|=2$, every non-empty
$\mathcal R_{8,S}$ is pure of codimension $6$ in $\mathcal U_8$.  For
$|S|=4$, after identifying $S$ with its complement, every non-empty
$\mathcal R_{8,S}$ is pure of codimension $8$.  Consequently
\[
\operatorname{codim}_{\mathcal U_8}\mathcal R_8=6,
\]
and every component of minimal codimension is generically of type $2+6$.
\end{conjecture}

\begin{remark}[Degree eight under Conjecture~\ref{con:reversal}]
\label{rem:d8-reduction}
Conjecture~\ref{con:reversal} would reduce the $63$ subset strata to seven:
the four reversal-invariant $2$-subsets
\[
\{1,8\},\qquad\{2,7\},\qquad\{3,6\},\qquad\{4,5\},
\]
and, up to complementation, the three reversal-invariant $4$-subsets
\[
\{1,2,7,8\},\qquad\{1,3,6,8\},\qquad\{1,4,5,8\}.
\]
All seven are realized at the spin point \eqref{eq:d8-spin-point}: by
\eqref{eq:d8-spin-factor} the quadratic factor $\lambda^2-q^2(1+t)$ with
$q=9-2i$ specializes at $t=0$ to $(\lambda+a_i)(\lambda+a_{9-i})$, and the
three unordered pairings of the four quadratics give the three $4$-subsets.
Thus under Conjecture~\ref{con:reversal} every non-empty degree-eight subset
stratum contains the spin point, and Conjecture~\ref{con:d8-class} becomes a
statement about exactly seven strata, each equipped with the exact transverse
test point of Proposition~\ref{prop:d8-ranks}.  For the endpoint stratum $\mathcal R_{8,\{1,8\}}$, the explicit family
$\mathcal N_8$ already has the predicted codimension.
\end{remark}

Propositions~\ref{prop:even-factor-degree},~\ref{prop:kappa}, and~\ref{prop:d8-ranks} separate the proved and conjectural parts of the
degree-eight picture.  The possible factor degrees are completely
determined, and Proposition~\ref{prop:kappa} gives codimension at most $6$
and $8$ for irreducible components of non-empty $2+6$ and $4+4$ subset
strata.  What remains open is the matching lower bound, i.e., the exclusion
of excess-dimensional components.  The spin point shows that both expected codimensions are
actually attained by transverse local branches.

\section{Outlook: a geometric codimension picture}\label{sec:outlook}

The results above suggest that, on the principal stratum, factor degree is the
natural geometric invariant of reducibility.  Proposition~\ref{prop:kappa}
gives an upper bound for the codimension of every non-empty subset stratum.
The main question is whether some components have excess dimension.

\begin{conjecture}[Expected codimension]\label{con:expected-codim}
Let $S\subset\{1,\ldots,n\}$ be non-empty and proper, put
$d=|S|\le n/2$, and assume $\mathcal R_{n,S}\ne\emptyset$.  Then every
irreducible component of $\mathcal R_{n,S}$ has codimension
\[
\kappa(n,d)=\left\lceil\frac{d(n-d)}2\right\rceil
\]
in $\mathcal U_n$.
\end{conjecture}

The conjectural content is the lower bound: Proposition~\ref{prop:kappa}
already gives $\operatorname{codim}\le\kappa(n,d)$.  Equality is proved in
degree four, and Proposition~\ref{prop:d8-ranks} gives transverse
degree-eight branches with the expected codimensions.  The endpoint family
realizes the expected value $n-2$ for $d=2$ in every degree.

For odd $n=2m+1$, the smallest formal value of $\kappa(n,d)$ is
$\kappa(n,1)=m$; Theorem~\ref{thm:central} shows that degree-one factors are
central constant branches.  For even $n$, odd factor degrees are excluded by
Proposition~\ref{prop:even-factor-degree}, and the smallest admissible value is
$\kappa(n,2)=n-2$.  This leads to the global profile.

\begin{conjecture}[Codimension of the reducible locus]
\label{con:codim-profile}
For $n\ge4$,
\[
\operatorname{codim}_{\mathcal U_n}\mathcal R_n=
\begin{cases}
\dfrac{n-1}{2},& n\ \mathrm{odd},\\[2mm]
n-2,& n\ \mathrm{even}.
\end{cases}
\]
For odd $n$, minimal components should be constant-branch components; for
even $n$, they should be generically of type $2+(n-2)$.
\end{conjecture}

The inequality ``$\le$'' is already proved: for even $n$ it follows from the
endpoint family, and for odd $n$ from the central branch in
Corollary~\ref{cor:spin-family}.  The quartic theorem proves equality for
$n=4$, while Proposition~\ref{prop:d8-ranks} supplies the expected transverse
branches in degree eight.

Three concrete problems remain.  First, prove
Conjecture~\ref{con:reversal}; already the degree-two case for general even
$n$ would reduce the degree-eight analysis from $63$ subset strata to seven.
Second, determine whether the closure of $\mathcal N_n$ is the only component
of the endpoint stratum $\mathcal R_{n,\{1,n\}}$ dominating the diagonal
space; Corollary~\ref{cor:endpoint-low} settles this for $n\le4$.  Third,
construct and classify the strata associated with the inner reversal pairs
$\{i,n+1-i\}$, and determine which of them meet the real Jacobi locus.
Proposition~\ref{prop:rank-symmetry} gives an independent necessary condition
in the real case.

In this picture the special symmetry mechanisms remain useful examples, but
the organizing principle is the geometry of factorization strata.  Degree
four is completely classified, degree eight has only the two splitting types
$2+6$ and $4+4$, and the conjectures above give a concrete framework for the
general case.

\appendix
\section{Exact degree-eight rank certificates}\label{app:exact-checks}

The only computational assertions used as results are the rank statements in
Proposition~\ref{prop:d8-ranks}; all calculations below are in exact integer
arithmetic.  For a factor of $\lambda$-degree $d$ we use
\[
F(\lambda,t)=\lambda^d+
\sum_{j=1}^{d}\left(\sum_{k=0}^{\lfloor j/2\rfloor}
f_{j,k}t^k\right)\lambda^{d-j},
\]
and the analogous template for the complementary factor.  We differentiate
all coefficients of $FG-P_8$ with respect to the chain and factor variables.

At the spin point \eqref{eq:d8-spin-point}, take
\[
F=\lambda^2-(1+t),\qquad
G=\prod_{q\in\{3,5,7\}}[\lambda^2-q^2(1+t)].
\]
The $2+6$ incidence Jacobian is $24\times33$.  The $24\times24$ minor formed
from all $15$ chain columns, all three $F$-coefficient columns, and
\[
g_{1,0},g_{2,0},g_{2,1},g_{3,0},g_{4,0},g_{4,1}
\]
has determinant
\[
-2^{83}3^{29}5^{13}7^{11}\cdot11\cdot23\cdot41\cdot12979\ne0.
\]
Hence the rank is $24$.  Exact row reduction gives the same rank for each of
the other three choices of quadratic factor.

For the $4+4$ grouping
\[
F=[\lambda^2-(1+t)][\lambda^2-9(1+t)],\qquad
G=[\lambda^2-25(1+t)][\lambda^2-49(1+t)],
\]
the Jacobian is $24\times31$.  Taking all $15$ chain columns, all eight
$F$-coefficient columns, and the column $g_{2,0}$ gives
\[
\det=-2^{102}3^{27}5^{18}7^{11}\ne0.
\]
The other two unordered pairings of the four quadratic factors also have
rank $24$ by exact row reduction.

\section*{Acknowledgements}
This project grew out of discussions with Professor B.~Mityagin at the
Department of Mathematics, Stockholm University, in 2010.  I am grateful to
Brendan Morey for identifying mistakes in an earlier version of the
manuscript through AI-assisted symbolic experimentation and for communicating
them to me.

\smallskip
\noindent\textit{Generative-AI disclosure.}
OpenAI ChatGPT (GPT-5.6 Sol) was used for language editing, organization of
the exposition, proof checking, and assistance in preparing exact symbolic
verification scripts.  Anthropic Claude was used in the same capacity for the
material of Sections~\ref{sec:central}--\ref{sec:factor-strata}, including
the symbolic experiments that located the endpoint family and the subset
scans, and assisted in drafting those sections.  The mathematical statements
are presented either with proofs or explicitly as conjectures or exact
computational checks.  The author takes responsibility for the content of the
manuscript.

\section*{Data availability statement}
No datasets were generated or analysed in this study.

\end{document}